\documentclass[leqno,preprint]{article}
\usepackage[english]{babel}
\usepackage[latin1]{inputenc}
\usepackage{amsmath}
\usepackage{amsfonts, amssymb,amsthm}
\usepackage{mathptmx}
\usepackage{url}
\usepackage{enumerate}
\usepackage[breaklinks=true, pdfpagemode=UseNone, pdfstartview={FitH}]{hyperref}
\usepackage[hyperpageref]{backref}
\newtheorem{theorem}{Theorem}[section]

\newtheorem{definition}{Definition}[section]
\newtheorem{proposition}{Proposition}[section]
\newtheorem{lemma}{Lemma}[section]

\newtheorem{remark}{Remark}[section]
\newtheorem{ex}{Example}[section]
\theoremstyle{definition}
\newtheorem{condition}{Condition}[section]
\newcommand{\lag}{\left \langle}
\newcommand{\rog}{\right \rangle}
\newcommand{\df}{\mathrel{\mathop:}=}

\usepackage{cleveref}
\crefname{condition}{Condition}{Conditions}
\crefname{ex}{Example}{Examples}
\crefname{remark}{Remark}{Remarks}
\author{Matteo Bianchi\\{\small Universit{\`a} degli Studi di Milano,}\\{\small Department of Computer Science}\\ {\small Via Comelico 39/41, 20135, Milano, Italy} \\ {\small\url{matteo.bianchi@unimi.it}}}
\date{}
\title{A temporal semantics for Nilpotent Minimum logic}
\begin{document}
\maketitle
\begin{abstract}
In \cite{ban} a connection among rough sets (in particular, pre-rough algebras) and three-valued {\L}ukasiewicz logic {\L}$_3$ is pointed out. In this paper we present a temporal like semantics for Nilpotent Minimum logic NM (\cite{fod,mtl}), in which the logic of every instant is given by {\L}$_3$: a completeness theorem will be shown. This is the prosecution of the work initiated in \cite{tmpg} and \cite{temp}, in which the authors construct a temporal semantics for the many-valued logics of Gödel (\cite{godint}, \cite{dum}) and Basic Logic (\cite{haj}).
\end{abstract}
\medskip
\begin{quote}
\textit{``Dedicated to a very important person: feel free to choose the rhythm and the direction of your walks, ever...''} 
\end{quote}	 
\section{Introduction and motivations}
Rough sets were introduced by Z. Pawlak in \cite{paw}, as an alternative to fuzzy sets. Triangular norms (see \cite{kmp}) are particular functions that can be used to define some operations in fuzzy sets theory or also as the semantical counterpart of the conjunction connective of a many-valued logic. For example, the book \cite{haj} introduces a formal framework of many-valued logics strictly connected with continuous t-norms and their residua: if a continuous t-norm can be seen as the semantical counterpart of a conjunction, the associated residuum plays the analogous role for the implication connective. Nilpotent minimum t-norm was introduced by J. Fodor in \cite{fod} as an example of a left-continuous but not continuous t-norm: left-continuous t-norms assume a particular relevance, since left-continuity is a necessary and sufficient condition for the existence of a residuum (\cite{beg}).

Nilpotent Minimum logic (NM) was introduced in \cite{mtl} as the logical system associated to the variety of algebras generated by $[0,1]_\text{NM}$, an algebraic structure induced by Nilpotent Minimum t-norm (see \Cref{subsec:sem} for definitions).

\medskip
\noindent NM has been the subject of various independent investigations. For example:
\begin{itemize}
\item Combinatorial aspects and description of the free algebras \cite{agm,abm,bus}.
\item States and connection with probability theory (\cite{ag}).
\item Computational complexity for satisfiability and tautologicity problems (\cite{nmcomp}).
\item Connections with others non-classical logics: for example, Nelson's constructive logic with strong negation CLSN (\cite{bc}). In particular, in \cite{bc} it is shown that NM is equivalent to CLSN plus the prelinearity axiom, that is \\\mbox{$(\varphi\to\psi)\vee(\psi\to\varphi)$}.
\item Extensions with truth constants, in the propositional and in the first-order case (\cite{egnfirst,egn,egn1,egn2}).
\end{itemize}
We now arrive at the relation among NM and rough sets. In \cite{ban} it is shown that every pre-rough algebra is term equivalent\footnote{Intuitively, this means that the operations of the two algebraic structures are ``inter-definable''. See \cite{mmt} for details.} to an algebra belonging to the variety generated by three elements Wajsberg-algebra $\mathbf{L}^w_3$. This last algebra is the semantical counterpart of three valued logic {\L}$_3$. 

In this paper we show that also NM is strictly connected with three valued \L ukasiewicz logic. We present a temporal like semantics for Nilpotent Minimum logic NM (\cite{fod,mtl}), in which the logic of every instant is given by {\L}$_3$, and the temporal flow is given by any totally ordered set: a completeness theorem will be shown. This provide a first connection among rough sets and Nilpotent Minimum logic.
\section{Preliminaries}\label{sec:1}
\subsection{Rough sets, pre-rough algebras, and three-valued Wajsberg-algebras}\label{sec:rough}
Rough sets were introduced by Z. Pawlak in \cite{paw}. The basic notion is the one of approximation space, which is a pair $\lag X,R \rog$, $X$ being a non-empty set (the domain of discourse) and $R$ an equivalence relation on it, representing an indiscernibility relation.
If $A\subseteq X$, the lower approximation $\underline{A}$ of $A$ in the approximation space $\lag X,R \rog$ is the union of equivalence classes contained in $A$, while its upper approximation $\overline{A}$ in $\lag X,R \rog$ is the union of equivalence classes properly intersecting $A$. $\underline{A}$ $(\overline{A})$ is interpreted as the collection of those objects of the domain X that definitely (possibly) belong to $A$.
A triple $\lag X,R,A \rog$, with $A\subseteq X$, is called a rough set. Pre-rough algebras, on the other side, are algebraic structures initially introduced in \cite{banc}, with the aim to find an algebraic framework to cope with rough sets. Here we mention the formulation of \cite{ban}. 

A pre-rough algebra is an algebraic structure with signature $\lag A,\leq,\sqcap,\sqcup,\neg,L,\Rightarrow,0,1\rog$ such that
\begin{itemize}
\item[P1]$\lag A,\leq,\sqcap,\sqcup,\neg,0,1\rog$ is a bounded distributive lattice with bottom element $0$ and top element $1$,
\item[P2]$\neg\neg a=a$,
\item[P3]$\neg(a\sqcup b)=\neg a\sqcap \neg b$,
\item[P4]$La\leq a$,
\item[P5]$L(a\sqcap b)=La\sqcap Lb$,
\item[P6]$LLa=La$,
\item[P7]$L1=1$,
\item[P8]$MLa=La$,
\item[P9]$\neg La\sqcup La=1$,
\item[P10]$L(a\sqcup b)=La\sqcup Lb$,
\item[P11]$La\leq Lb$, $Ma\leq Mb$ imply $a\leq b$,
\item[P12]$a\Rightarrow b=(\neg La\sqcup Lb)\sqcap(\neg Ma\sqcup Mb)$,
\end{itemize}
where $Ma\df\neg L\neg a$, and $a,b\in A$.

\bigskip
\noindent Wajsberg algebras were initially introduced in \cite{frt} as a counterpart of infinite-valued \L ukasiewicz logic: they form an algebraic variety (in the sense of universal algebra, see \cite{mmt}). 

\noindent A particular subvariety of them is the one generated by the algebra $\mathbf{L^W_3}=\lag \left\{0,\frac{1}{2},1\right\},\rightarrow,\neg,1\rog$, where $x\rightarrow y=\min(1,1-x+y)$ and $\neg x=1-x$, for every $x,y\in \left\{0,\frac{1}{2},1\right\}$. We call this variety of algebras three valued Wajsberg algebras. 

In general, a three valued Wajsberg algebra is a structure of the form $\lag A,\rightarrow,\neg,1\rog$ such that, for every $a,b,c\in A$
\begin{align}
\tag	{w1}&a\to(b\to a)=1\\
\tag	{w2}&(a\to b)\to ((b\to c)\to(a\to c))=1\\
\tag	{w3}&((a\to \neg a)\to a)\to a=1\\
\tag	{w4}&(\neg a\to \neg b)\to (b\to a)=1
\end{align}
The result that we mentioned in the introduction is the following:
\begin{theorem}[\cite{ban}]\label{teo:l3eq}
Pre-rough algebras are term-equivalent to three valued Wajsberg algebras.
\end{theorem}
The connection about $\mathbf{L^W_3}$ and three-valued \L ukasiewicz logic (introduced by \L ukasiewicz in \cite{triv} and axiomatized by Wajsberg in \cite{waj}) \L$_3$ was initially proved by Wajsberg in \cite{waj}, in which he showed a completeness theorem. For this reason \Cref{teo:l3eq} connects three valued \L ukasiewicz logic and rough sets. To conclude, we list the axioms of \L$_3$: its formulas are defined in the usual way from a set of denumerable variables and the connectives $\neg,\to$.
\begin{align}
\tag	{\L1}&\varphi\to(\psi\to \varphi)\\
\tag	{\L2}&(\varphi\to \psi)\to ((\psi\to \chi)\to(\varphi\to \chi))\\
\tag	{\L3}&((\varphi\to \neg \varphi)\to \varphi)\to \varphi\\
\tag	{\L4}&(\neg \varphi\to \neg \psi)\to (\psi\to \varphi)
\end{align}
\subsection{Many-valued logics: syntax}\label{sec:synt}
Monoidal t-norm based logic (MTL) was introduced in \cite{mtl}: it is based over connectives $\&, \land, \to, \bot$ (the first three are binary, whilst the last one is $0$-ary). The notion of formula is defined inductively starting from the fact that all propositional variables (we will denote their set with $VAR$) and $\bot$ are formulas. The set of all formulas will be called $FORM$.

\noindent Useful derived connectives are the following
\begin{align}
\tag{negation}\neg\varphi\df& \varphi\to\bot\\
\tag{disjunction}\varphi\vee\psi\df& ((\varphi\to\psi)\to\psi)\land((\psi\to\varphi)\to\varphi)\\
\tag{strong disjunction}\varphi\veebar\psi\df&\neg(\neg\varphi\&\neg\psi).
\end{align}
For the reader's convenience we list the axioms of MTL
\begin{align}
\tag{A1}&(\varphi \rightarrow \psi)\rightarrow ((\psi\rightarrow \chi)\rightarrow(\varphi\rightarrow \chi))\\
\tag{A2}&(\varphi\&\psi)\rightarrow \varphi\\
\tag{A3}&(\varphi\&\psi)\rightarrow(\psi\&\varphi)\\
\tag{A4}&(\varphi\land\psi)\rightarrow \varphi\\
\tag{A5}&(\varphi\land\psi)\rightarrow(\psi\land\varphi)\\
\tag{A6}&(\varphi\&(\varphi\rightarrow \psi))\rightarrow (\psi\land\varphi)\\
\tag{A7a}&(\varphi\rightarrow(\psi\rightarrow\chi))\rightarrow((\varphi\&\psi)\rightarrow \chi)\\
\tag{A7b}&((\varphi\&\psi)\rightarrow \chi)\rightarrow(\varphi\rightarrow(\psi\rightarrow\chi))\\
\tag{A8}&((\varphi\rightarrow\psi)\rightarrow\chi)\rightarrow(((\psi\rightarrow\varphi)\rightarrow\chi)\rightarrow\chi)\\
\tag{A9}&\bot\rightarrow\varphi
\end{align}
As inference rule we have modus ponens:
\begin{equation}
\tag{MP}\frac{\varphi\quad \varphi\rightarrow\psi}{\psi}
\end{equation}
Nilpotent Minimum Logic (NM), introduced in \cite{mtl} is obtained from MTL by adding the following axioms:
\begin{align}
\tag{involution}&\neg\neg\varphi\to\varphi\\
\tag{WNM}&\neg(\varphi\&\psi)\vee((\varphi\land\psi)\to(\varphi\&\psi))
\end{align}
The previously mentioned \L ukasiewicz three-valued logic \L$_3$ can also be axiomatized as MTL plus (see \cite{haj,mtl})
\begin{align}
\tag{involution}&\neg\neg\varphi\to\varphi\\
\tag	{div}&(\varphi\land(\varphi\rightarrow \psi))\rightarrow (\psi\&\varphi)\\
\tag{$c_3$}&\varphi^2\to\varphi^3
\end{align} 
The notions of theory, syntactic consequence, proof are defined as usual.
\subsection{Many-valued logics: semantics}\label{subsec:sem}
\bigskip
An MTL algebra is an algebra $\lag A,*,\Rightarrow,\sqcap,\sqcup,0,1\rog$ such that
\begin{enumerate}
\item $\lag A,\sqcap,\sqcup, 0,1\rog$ is a bounded lattice with bottom $0$ and top $1$.
\item $\lag A,*,1 \rog$ is a commutative monoid.
\item $\lag *,\Rightarrow \rog$ forms a \emph{residuated pair}: $z*x\leq y$ iff $z\leq x\Rightarrow y$ for all $x,y,z\in A$.
\item The following axiom hold, for all $x,y\in A$:
\begin{equation}
\tag{Prelinearity}(x\Rightarrow y)\sqcup(y\Rightarrow x)=1
\end{equation}
A totally ordered MTL-algebra is called MTL-chain.
\end{enumerate}
An MV-algebra is an MTL-algebra that satisfies the following equations:
\begin{align*}
\sim\sim x&=x\\
x\sqcap y&=x*(x\Rightarrow y)
\end{align*}
Where $\sim x$ indicates $x\Rightarrow 0$.

\bigskip
The connection among Wajsberg and MV-algebras is the following:
\begin{theorem}[\cite{frt,haj}]\label{teo:mveq}
MV-algebras are term equivalent to Wajsberg algebras.
\end{theorem}
\smallskip
\noindent An NM-algebra is an MTL-algebra that satisfies the following equations:
\begin{align*}
&\sim\sim x=x\\
&\sim(x*y)\sqcup((x\sqcap y)\Rightarrow(x*y))=1
\end{align*}
\smallskip
Moreover, as noted in \cite{gis}, in \emph{each} NM-chain it holds that:
\begin{align*}
x*y=&
\begin{cases}
0&\text{if }x\leq n(y)\\
\min(x,y)&\text{Otherwise.}
\end{cases}\\
x\Rightarrow y=&
\begin{cases}
1&\text{if }x\leq y\\
\max(n(x),y)&\text{Otherwise.}
\end{cases}
\end{align*}
Where $n$ is a strong negation function, i.e. $n:A\to A$ is an order-reversing mapping ($x\leq y$ implies $n(x)\geq n(y)$) such that $n(0)=1$ and $n(n(x))=x$, for each $x\in A$. Observe that $n(x)=x\Rightarrow 0$, for each $x\in A$.

If we define $x\oplus y\df\sim(\sim x*\sim y)$ (this is the algebraic counterpart of the connective $\veebar$), then an easy check shows that
\begin{equation*}
x\oplus y=
\begin{cases}
1&\text{if }n(x)\leq y\\
\max(x,y)&\text{Otherwise.}
\end{cases}
\end{equation*}
A negation fixpoint is an element $x\in A$ such that $n(x)=x$: note that if exists then it must be unique (otherwise $n$ fails to be order-reversing).
A positive element is an $x\in A$ such that $x>n(x)$; the definition of negative element is the dual. Given an NM-chain with support $A$, the set of positive (negative) elements will be denoted by $A^+$ ($A^-$).
\begin{ex}\label{ex:l3}
Consider the MTL-algebra	$\mathbf{L_3}=\lag\left\{0,\frac{1}{2},1\right\},\min,\max,*\rightarrow,0,1\rog$, where $x*y=\max(0,x+y-1)$ and $x\Rightarrow y=\min(1,1-x+y)$, for every $x,y\in \left\{0,\frac{1}{2},1\right\}$. A direct inspection shows that it is an MV-algebra and also an NM-algebra: in particular in can be shown that every three elements MV-algebra (NM-algebra) is isomorphic to it.

\noindent Note also that, in the light of \Cref{teo:mveq}, $\mathbf{L_3}$ is term-equivalent to the three element Wajsberg $\mathbf{L^w_3}$.
\end{ex}
\begin{definition}	\label{sem:as}
Let $\mathcal{A}$ be an NM-algebra.
Each map $e\colon \ VAR\to A$ extends uniquely to an $\mathcal{A}$-{\em assignment} $v_e:\ FORM\to A$, by
the following inductive prescriptions:
\begin{itemize}
\item $v_e(\bot)=0$
\item $v_e(\varphi\to\psi)=v_e(\varphi)\Rightarrow v_e(\psi)$
\item $v_e(\varphi\&\psi)=v_e(\varphi)* v_e(\psi)$
\item $v_e(\varphi\land\psi)=v_e(\varphi)\sqcap v_e(\psi)$
\end{itemize}
\end{definition}
A formula $\varphi$ is {\em consequence} of a theory (i.e. set of formulas) $\Gamma$ in an NM-algebra $\mathcal{A}$, in symbols, $\Gamma\models_\mathcal{A}\varphi$,
if for each $\mathcal{A}$-assignment $v$, $v(\psi)=1$ for all $\psi\in \Gamma$ implies that $v(\varphi)=1$.
\subsection{Completeness}
\noindent Let $\mathcal{A}$ be an NM-chain. We say that NM is strongly complete (respectively: finitely strongly complete, complete) with respect to $\mathcal{A}$
if for every theory $\Gamma$ (respectively, for every finite theory $\Gamma$ of formulas, for
$\Gamma=\emptyset$) and for every formula $\varphi$ we have
\begin{equation*}
\Gamma\vdash_\text{NM}\varphi\quad \text{iff}\quad \Gamma\models_\mathcal{A}\varphi
\end{equation*}
\begin{theorem}[\cite{mtl,gis}]
Let $\mathcal{A}$ be an infinite NM-chain with negation fixpoint. Then NM is complete w.r.t. $\mathcal{A}$.
\end{theorem}
\noindent This result can be improved:
\begin{theorem}\label{teo:fcomp}
Let $\mathcal{A}$ be an infinite NM-chain with negation fixpoint. Then NM is finitely strongly complete w.r.t. $\mathcal{A}$.
\end{theorem}
\begin{proof}
In \cite[Theorem 3.8]{dist} is shown that this property is equivalent to ask that every NM-chain $\mathcal{B}$ is partially embeddedable into $\mathcal{A}$. We will show that each finite subalgebra of an NM-chain embeds in $\mathcal{A}$: this is enough, since every finitely-generated NM-chain is finite\footnote{This is not difficult to check: if we start with a generating set with $n$ elements, then its closure has cardinality at most $2n+2$, since we can add $0,1$ and the negations of the elements. In general, the variety of NM-algebras is locally finite.}.

Take an NM-chain $\mathcal{B}$, and let $\mathcal{X}$ be a finite subalgebra of $\mathcal{B}$ with universe $X$. Construct a map $\phi:\mathcal{X}\to\mathcal{A}$ such that 
\begin{itemize}
\item $\phi(0)=0, \phi(1)=1$.
\item The elements of $X^-$ are mapped in $A^-$ preserving the order, that is for every $x,x'\in X^-$ such that $x<x'$ it holds that $\phi(x),\phi(x')\in A^-$ with $\phi(x)<\phi(x')$. 
\item If $\mathcal{X}$ has a negation fixpoint $f_\mathcal{X}$, then $\phi(f_\mathcal{X})=f_\mathcal{A}$ ($f_\mathcal{A}$ being the negation fixpoint of $\mathcal{A}$).
\item For every $x\in X^+$ we set $\phi(x)=n_\mathcal{A}(\phi(n_\mathcal{X}(x)))$. 
\end{itemize}
Direct inspection shows that $\phi$ is an embedding: this concludes the proof.
\end{proof}

\noindent Here we list some examples of infinite NM-chains with negation fixpoint:

\smallskip
\begin{itemize}
\item $NM_\infty=\lag \{a_n\}_{n\in \mathbb{Z}}\cup\{0,1\},*,\Rightarrow,\min,\max,0,1\rog$, where
\begin{itemize}
\item $0<a_n<1$ for all $n\in\mathbb{Z}$; for all $n,m\in\mathbb{Z}$, if $n<m$, then $a_n<a_m$.
\item $n(0)=1, n(1)=0$ and, for all $m\in\mathbb{Z}$, $n(a_m)=a_{0-m}$.
\end{itemize}
\item $[0,1]_{NM}=\lag [0,1],*,\Rightarrow,\min,\max,0,1\rog$, where
\begin{itemize}
\item the order is given by $\leq_\mathbb{R}$.
\item $n(x)=1-x$, for each $x\in [0,1]$.
\end{itemize}
\item $[0,1]^\mathbb{Q}_{NM}=\lag [0,1]\cap\mathbb{Q},*,\Rightarrow,\min,\max,0,1\rog$, where
\begin{itemize}
\item the order is given by $\leq_\mathbb{Q}$.
\item $n(x)=1-x$, for each $x\in [0,1]\cap\mathbb{Q}$.
\end{itemize}
\end{itemize}
Concerning $[0,1]_{NM}$ and $[0,1]^\mathbb{Q}_{NM}$, we have a that:
\begin{theorem}[\cite{mtl,dist}]\label{theorem:nmstd}
NM enjoys the strong completeness with respect to $\mathcal{A}$, with $\mathcal{A}\in \left\{[0,1]_{NM}, [0,1]^\mathbb{Q}_{NM}\right\}$.
\end{theorem}
\begin{remark}\label{rem:comp}
Note that NM does not enjoy the strong completeness w.r.t. $NM_\infty$. Indeed, due to \cite[Theorem 3.5]{dist} this is equivalent to ask that every countable NM-chain can be embedded into $NM_\infty$: however, this does not hold. For example, $[0,1]^\mathbb{Q}_{NM}\not\hookrightarrow NM_\infty$, for the same reason for which rational numbers cannot be embedded into integer numbers, by preserving the order.
\end{remark}

\section{A temporal semantics for Nilpotent Minimum logic}
In this section we discuss a temporal semantics related to NM. 

In the approach that we present, the temporal flow is an arbitrary totally ordered infinite set $\lag T,\leq\rog$: the elements of $T$ are called \emph{instants}. In particular, the logic associated to the single instant (or world, by using a well-established terminology in modal logic) is based over three truth-values, $\left\{0,\frac{1}{2},1\right\}$, ordered in the way that $0<\frac{1}{2}<1$.

Over these three values we can define the semantics associated to a negation and an implication operations:
\begin{table}[h]
\begin{center}
\begin{tabular}{|c|c|}
\hline
 $\neg_3$   &  \\
\hline
$0$  & $1$   \\
\hline
$\frac{1}{2}$  & $\frac{1}{2}$  \\
\hline
$1$  & $0$  \\
\hline
\end{tabular}
\hspace{2 cm}
\begin{tabular}{|c|c|c|c|}
\hline
 $\rightarrow_3$   & $0$ & $\frac{1}{2}$ & $1$ \\
\hline
$0$  & $1$ & $1$ & $1$  \\
\hline
$\frac{1}{2}$  & $\frac{1}{2}$ & $1$ & $1$ \\
\hline
$1$  & $0$ & $\frac{1}{2}$ & $1$ \\
\hline
\end{tabular}
\end{center}
\caption{Trivalent operations}
\label{tab:1}
\end{table}
in particular, note that the negation is involutive, and hence its behavior is substantially determined by the order of the truth values previously specified. Concerning the implication, it is a generalization of classical implication: we have that $a\rightarrow_3 b=1$ whenever $a\leq b$, that $1\rightarrow_3 b=b$, and that $a\rightarrow_3 0$ has the same behavior of $\neg_3 a$ (that is, $\neg_3$ is obtained by pseudocomplementation from $\rightarrow_3$). 
\begin{remark}	
It can be noted that these operations are the same of the algebra $\mathbf{L_3^w}$, that - as we have seen in \Cref{sec:rough} - gives an algebraic semantics for \L ukasiewicz three valued logic \L$_3$.
\end{remark}
In the proposed semantics a temporal assignment (over a temporal flow $\lag T,\leq\rog$) is a function $v:\ FORM\times T\to \left\{0,\frac{1}{2},1\right\}$. However, not arbitrary assignments are admitted: in our semantics we restrict to three typologies of temporal assignments, as indicated below:
\begin{condition}\label{cond:1}
We restrict to the following types of temporal assignments $v:\ FORM\times T\to \left\{0,\frac{1}{2},1\right\}$, for every $\varphi\in FORM$:
\begin{enumerate}
\item $v(\varphi,\cdot)$ is constant, to $0$, $\frac{1}{2}$ or $1$. 

In this case we say that, respectively $v(\varphi,\cdot)\approx 0$, $v(\varphi,\cdot)\approx \frac{1}{2}$, $v(\varphi,\cdot)\approx 1$. 
\item There is a $t\in T$, with $t\neq \min(T)$ (if $T$ has a minimum) such that
\begin{equation*}
v(\varphi,t')=0 \text{ for every }t'\geq t,\text{ and }v(\varphi,t'')=\frac{1}{2},\text{ for every }t''<t.
\end{equation*}
In this case we say that $v(\varphi,\cdot)\approx t_0$.
\item There is a $t\in T$, with $t\neq \min(T)$ (if $T$ has a minimum) such that
\begin{equation*}
v(\varphi,t')=1 \text{ for every }t'\geq t,\text{ and }v(\varphi,t'')=\frac{1}{2},\text{ for every }t''<t.
\end{equation*}
In this case we say that $v(\varphi,\cdot)\approx t_1$.
\end{enumerate}
\end{condition}
The following lemma guarantees that $\approx$ associates no more than a value to a temporal assignment (applied to a formula). 
\begin{lemma}\label{lem:cldisg}
The symbol $\approx$ introduced in \Cref{cond:1} is a partial map.
\end{lemma}
\begin{proof}
An easy check.
\end{proof}
If we think at $0$ as ``false'', and at $1$ as ``true'', then $\frac{1}{2}$ can be understood as an intermediate state among them. Note that we have, essentially, five different types of temporal assignments: the characteristic that all of them have in common (by the only exception of the assignment constant to $\frac{1}{2}$) is that they assume a stable value  $0$ or $1$, after some period of time. The first three types are the constant ones, to $0$, $\frac{1}{2}$ or $1$: the meaning of the first and the last one is clear, one makes a formula always false, and the other always true. The assignment constant to $\frac{1}{2}$, instead, is more delicate: a possible interpretation, if we are interested only to the ``definitive'' truth or the falsity of a formula, is that this last one never reaches a well definite truth-value.

Finally, concerning the remaining two types of assignments, for both of them there is an instant starting from that the formula assumes a definite truth-value (true or false), and before that it is in an intermediate state.
 
Now we introduce the definition of temporal assignment (first on the variables, and then we will extend it over formulas):
\begin{definition}\label{def:var}
A temporal assignment over variables (associated to a temporal flow $\lag T,\leq\rog$) is a function $v:\ VAR\times T\to \left\{0,\frac{1}{2},1\right\}$ such that one of the following holds, for every $x\in VAR$:
\begin{itemize}
\item $v(x,\cdot)$ is constant.
\item There is a $t\in T$, with $t\neq \min(T)$ (if $T$ has a minimum) such that 
\begin{equation*}
v(\varphi,t')=0 \text{ for every }t'\geq t,\text{ and }v(\varphi,t'')=\frac{1}{2},\text{ for every }t''<t.
\end{equation*}
\item There is a $t\in T$, with $t\neq \min(T)$ (if $T$ has a minimum) such that 
\begin{equation*}
v(\varphi,t')=1 \text{ for every }t'\geq t,\text{ and }v(\varphi,t'')=\frac{1}{2},\text{ for every }t''<t.
\end{equation*}
\end{itemize}
\end{definition}
We now extend our notion of temporal assignments to the formulas of Nilpotent Minimum logic
\begin{remark}
We will consider only $\to,\neg$, as connectives. This is because, as pointed out in \cite{maxdf}, in Nilpotent Minimum logic the disjunction $\vee$ is definable from $\neg,\to$ as $\varphi\vee\psi\df(\varphi\to(\varphi\to\psi))\&((\varphi\to(\varphi\to\psi))\to(\psi\to(\psi\to\varphi)))$, where $\varphi\&\psi\df\neg(\varphi\to\neg\psi)$, and $\varphi\land\psi\df\neg(\neg\varphi\vee\neg\psi)$.
\end{remark}
In the semantics we want to describe we are essentially interested to the stable behavior of the assignments, that is when they assume a constant value, for all the future instants of time. 
\begin{definition}\label{def:form1}
Let $v$ be a temporal assignment over variables, associated to some temporal flow $T$. Its extension $v':\ FORM\times T\to \left\{0,\frac{1}{2},1\right\}$ to formulas is defined, inductively, in the following way, for every $\varphi\in FORM$, and $t\in T$:
\begin{equation*}
v'(\varphi,t)\df
\begin{cases}
v(x,t)&\text{if }\varphi=x\\
0&\text{if }\varphi=\bot\\
\neg_3v'(\psi,t)&\text{if }\varphi=\neg\psi\\
v'_d(\psi\to\chi,t)&\text{if }\varphi=\psi\to\chi.
\end{cases}
\end{equation*}
Where $\psi,\chi\in FORM$, $x\in VAR$ and
\begin{equation*}
v'_d(\psi\to\chi,t)\df
\begin{cases}
v'(\psi,t)\to_3 v'(\chi,t)&\text{if }v'(\psi,t)\to_3 v'(\chi,t)=v'(\psi,t')\to_3 v'(\chi,t'),\\
&\text{for every }t'\geq t\\
\frac{1}{2}&\text{otherwise.}
\end{cases}
\end{equation*}
\end{definition}
Essentially, we are applying the operations described in \Cref{tab:1} in a ``pointwise'' way, that is instant by instant. The function $v_d$ associates to an assignment $v$ its ``definitive behavior'': this function is necessary to restrict ourself on the assignments of \Cref{cond:1}, as the following proposition shows. 
\begin{lemma}\label{lem:def}
The temporal assignments of \Cref{def:form1} satisfy \Cref{cond:1}.
\end{lemma}
\begin{proof}
By induction over $\varphi$.

If $\varphi\in VAR\cup\left\{\bot\right\}$, then the result easily follows from \Cref{def:var,def:form1}.

Suppose that $\varphi\df\neg\psi$ and that the claim holds for $\psi$: since for every temporal assignment $v$ and instant $t$, $v(\varphi,t)=1-v(\psi,t)$, then by the inductive hypothesis and an easy check we have the result.

Finally, suppose that $\varphi\df\psi\to\chi$ and that the claim holds for $\psi,\chi$. Take a temporal assignment $v$. A direct inspection shows that:
\begin{itemize}
\item If $v(\psi,\cdot)\approx 0$, then $v(\varphi\to\psi,t)=v(\neg\varphi,t)$, for every $t$.
\item If $v(\psi,\cdot)\approx 1$ or $v(\varphi,\cdot)\approx 0$, then $v(\varphi\to\psi,\cdot)\approx 1$.
\item If $v(\varphi,\cdot)\approx 1$ then $v(\varphi\to\psi,\cdot)\approx v(\psi,\cdot)$.
\item If $v(\varphi,\cdot)\approx t_1$ and $v(\psi,\cdot)\approx t'_1$, with $t\geq t'$ then $v(\varphi\to\psi,\cdot)\approx 1$.
\item If $v(\varphi,\cdot)\approx t_0$ and $v(\psi,\cdot)\approx t'_0$, with $t\leq t'$ then $v(\varphi\to\psi,\cdot)\approx 1$.
\item If $v(\varphi,\cdot)\approx t_0$ and $v(\psi,\cdot)\approx t'_1$, then $v(\varphi\to\psi,\cdot)\approx 1$.
\item If $v(\varphi,\cdot)\approx t_1$ and $v(\psi,\cdot)\approx t'_1$, with $t<t'$ then $v(\varphi\to\psi,\cdot)\approx t'_1$.
\item If $v(\varphi,\cdot)\approx t_0$ and $v(\psi,\cdot)\approx t'_0$, with $t'<t$ then $v(\varphi\to\psi,\cdot)\approx t_1$.
\item If $v(\varphi,\cdot)\approx t_1$ and $v(\psi,\cdot)\approx t'_0$, then $v(\varphi\to\psi,\cdot)\approx \max(t,t')_0$.
\item If $v(\varphi,\cdot)\approx \frac{1}{2}$ and $v(\psi,\cdot)\approx t_1$, then $v(\varphi\to\psi,\cdot)\approx 1$.
\item If $v(\varphi,\cdot)\approx \frac{1}{2}$ and $v(\psi,\cdot)\approx t_0$, then $v(\varphi\to\psi,\cdot)\approx \frac{1}{2}$.
\item If $v(\varphi,\cdot)\approx t_0$ and $v(\psi,\cdot)\approx \frac{1}{2}$, then $v(\varphi\to\psi,\cdot)\approx 1$.
\item If $v(\varphi,\cdot)\approx \frac{1}{2}$ and $v(\psi,\cdot)\approx \frac{1}{2}$, then $v(\varphi\to\psi,\cdot)\approx 1$.
\item If $v(\varphi,\cdot)\approx t_1$ and $v(\psi,\cdot)\approx \frac{1}{2}$, then $v(\varphi\to\psi,\cdot)\approx \frac{1}{2}$.
\end{itemize}
This exhausts all the cases and concludes the proof.
\end{proof}
Note that, by analyzing the question from a different perspective, a temporal assignment is a function that associates to every formula a certain sequence (indexed by the instants of time) of truth-values:
\begin{definition}\label{def:bitseq}
Given a temporal assignment $v$ (over a temporal flow $\lag T,\leq\rog$), one can define a function $\cdot^v$ from the set of formulas into the set of sequences of $\left\{0,\frac{1}{2},1\right\}^T$ by
\begin{equation*}
\varphi^v\df v(\varphi,\cdot).
\end{equation*}
We set $\mathcal{T}_T=\left\{\varphi^v:\, \varphi\text{ is a formula and }v\text{ is a temporal assignment over }\lag T,\leq\rog\right\}$.
\end{definition}
As shown in \Cref{def:form1} , for every temporal assignment $v$ (over a temporal flow $T$) and formula $\varphi$, the sequence $\varphi^v$ has one of the three types of behavior described in \Cref{cond:1}. Since we are interested in the definitive behavior of a temporal assignment, we now define an operator that ``capture'' the behavior of an assignment.
\begin{definition}\label{def:defbeh}
Let $\varphi, v$ be a formula and a temporal assignment over a temporal flow $\lag T,\leq\rog$, and let $T'=T\cup \{-\infty\}$. The definitive behavior operator $d:\mathcal{T}_T\to T'\times\left\{0,\frac{1}{2},1\right\}$ is defined as follows:
\begin{itemize}
\item $d(\varphi^v)=\lag -\infty, 1\rog$ if $\varphi^v\approx 1$.
\item $d(\varphi^v)=\lag -\infty, 0\rog$ if $\varphi^v\approx 0$.
\item $d(\varphi^v)=\lag -\infty,\frac{1}{2}\rog$ if $\varphi^v\approx \frac{1}{2}$.
\item $d(\varphi^v)=\lag t,1\rog$ if $\varphi^v\approx t_1$.
\item $d(\varphi^v)=\lag t,0\rog$ if $\varphi^v\approx t_0$.
\end{itemize}
The fact that $d$ is a well defined map is assured by \Cref{lem:def}.
\end{definition}
\begin{remark}
The first component of the pairs $\lag t,i\rog$ indicates the instant of time in which the function $\varphi^v$ assumes the ``stable'' value: this last one ($0$, $\frac{1}{2}$ or $1$) is specified in the second component. This justify the fact that $\lag-\infty,i\rog$ indicates that the function assumes always the value $i$.
\end{remark}
\begin{definition}	
Let $T'=T\cup \{-\infty\}$. We define a total order relation $\leq_{T''}$, over $T''=T'\times \{0,1\}\cup\left\{\lag -\infty,\frac{1}{2}\rog\right\}$, as follows:
\begin{itemize}
\item for each $t,t'\in T$, with $t<t'$, $\lag-\infty,0\rog<_{T''}\lag t,0\rog<_{T''}\lag t',0\rog<_{T''}\lag -\infty,\frac{1}{2}\rog<_{T''}\lag t',1\rog<_{T''}\lag t,1\rog<_{T''}\lag-\infty,1\rog$.
\end{itemize}
\end{definition}
Now:
\begin{definition}\label{defform}
For each temporal assignment $v$ (over a temporal flow $\lag T,\leq\rog$) the function \mbox{$s^v:\ FORM\to T''$} has the following behavior:
\begin{itemize}
\item $s^v(x_i)=d(x_i^v)$.
\item $s^v(\bot)=\lag -\infty,0\rog$.
\item If $s^v(\varphi)=\lag a,n\rog$ and $s^v(\psi)=\lag b,n'\rog$, then
\begin{align*}
s^v(\neg\varphi)&=\lag a,1-n\rog\\
s^v(\varphi\to\psi)&=
\begin{cases}
\lag-\infty,1\rog&\text{If }\lag a,n\rog\leq_{T''}\lag b,n'\rog\\
s^v(\neg\varphi)\curlyvee s^v(\psi)&\text{Otherwise}
\end{cases}
\end{align*}
Where $\curlyvee$ denotes the maximum over $\leq_{T''}$.
\end{itemize}
\end{definition}
It is immediate to check that
\begin{proposition}\label{prop:inv}
For each formula $\varphi$, and temporal assignment $v$ it holds that:
\begin{equation*}
s^v(\neg\neg\varphi)=s^v(\varphi).
\end{equation*}
\end{proposition}
The following theorem shows that \Cref{def:form1} and \Cref{defform} are equivalent, from the point of view of the ``definitive behavior'' of an assignment.
\begin{theorem}
Let $v$ be a temporal assignment. For every formula $\varphi$ it holds that 
\begin{equation*}
s^v(\varphi)=d(\varphi^v).
\end{equation*}
\end{theorem}
\begin{proof}
A direct inspection from \Cref{def:form1,defform} and \Cref{lem:def}.
\end{proof}
\section{Completeness}
In this section we show that the temporal semantics previously introduced is complete w.r.t. the logic NM.

\medskip
\noindent We begin with the following proposition:
\begin{proposition}\label{prop:isoneg}
Given a temporal flow $\lag T,\leq \rog$ there is an NM-chain $\mathcal{A}_T$, with negation fixpoint $\lag -\infty,\frac{1}{2}\rog$, whose lattice reduct is $\lag T'',\leq_{T''}\rog$.
\end{proposition}
\begin{proof}
From the results of \Cref{subsec:sem}, we know that the operations of an NM-chain are defined over the order and the negation function. Hence it remains only to show a strong negation function $n$ over $T''$. Since every element of $x$ of $T''$ has the form $\lag a,b\rog$, with $a\in T'$, and $b\in\left\{0,\frac{1}{2},1\right\}$, then an easy check shows that the function $n:\, T''\to T''$ such that
\begin{equation*}
n(\lag a,b\rog)=\lag a,1-b\rog
\end{equation*}
is a strong negation function over $T''$.
\end{proof}
\begin{proposition}\label{prop:tmpas}		
Let $\lag T,\leq\rog$ be a temporal flow and $\varphi$ be a formula. For every temporal assignment $v$ there is an $\mathcal{A}_T$-assignment $v'$ such that $s^v(\varphi)=v'(\varphi)$; conversely, for every $\mathcal{A}_T$-assignment $w$ there is a temporal assignment $w'$ over $\lag T,\leq\rog$ such that $s^{w'}(\varphi)=w(\varphi)$.
\end{proposition}
\begin{proof}
Let $v$ be a temporal assignment over $\lag T,\leq\rog$: take an $\mathcal{A}_T$-assignment $v'$ such that $v'(x)=s^v(x)$, for every variable $x$ of $\varphi$. This can be done, since the support of $\mathcal{A}_T$ is $T''$. The fact that $v'(\varphi)=s^v(\varphi)$ follows immediately by inspecting \Cref{sem:as,defform} and \Cref{prop:isoneg}, since the semantics given to the various connectives is the same, for both the assignments.

Conversely, with an analogous argument, given an $\mathcal{A}_T$-assignment $w$ we can find a temporal assignment $w'$ over $\lag T,\leq\rog$ such that $w(\varphi)=s^{w'}(\varphi)$.
\end{proof}
\begin{theorem}\label{teo:nmtmp}
Let $\lag T,\leq\rog$ be a temporal flow. Then, for every formula $\varphi$ and theory $\Gamma$ it holds that
\begin{equation*}
\Gamma \models_T \varphi\qquad\text{iff}\qquad\Gamma\models_{\mathcal{A}_T}\varphi,
\end{equation*}
where $\Gamma \models_T \varphi$ means that for every temporal assignment $v$ such that $s^v(\psi)=\lag -\infty,1\rog$ for every $\psi\in\Gamma$, it holds that $s^v(\varphi)=\lag -\infty,1\rog$.
\end{theorem}
\begin{proof}
Immediate from \Cref{prop:tmpas}.
\end{proof}
We finally obtain
\begin{theorem}[Completeness theorem]\label{teo:comp}
Let $\lag T,\leq\rog$ be a temporal flow. Then for each formula $\varphi$ and finite theory $\Gamma$
\begin{equation*}
\Gamma\vdash_{NM}\varphi\quad\text{iff}\quad \Gamma\models_T\varphi.
\end{equation*}
\end{theorem}
\begin{proof}
Let $\lag T,\leq\rog$ be a temporal flow: as shown in \Cref{prop:isoneg} $\mathcal{A}_T$ is an infinite NM-chain with negation fixpoint. From \Cref{teo:fcomp} we have that NM is finitely strongly complete w.r.t. $\mathcal{A}_T$, and by applying \Cref{teo:nmtmp} we show the claim of the theorem.
\end{proof}
Now we conclude by giving some examples of temporal flows connected to interesting NM-chains:
\begin{ex}
Let $\lag T,\leq\rog=\lag\mathbb{N},\geq_\mathbb{N}\rog$: it follows that $\mathcal{A}_T\simeq NM_\infty$. Thanks to \Cref{teo:comp} we have that, for each formula $\varphi$ and each \emph{finite} theory $\Gamma$:
\begin{equation*}
\Gamma\vdash_{NM}\varphi\quad\text{iff}\quad \Gamma\models_{\mathcal{A}_T}\varphi.
\end{equation*}
Note, however, that the strong completeness does not hold: this is a consequence of \Cref{rem:comp}, and \Cref{teo:nmtmp}.

Consider now $\lag T,\leq\rog\in\left\{\lag\mathbb{R},\leq_\mathbb{R}\rog,\lag\mathbb{Q},\leq_\mathbb{Q}\rog\right\}$: it follows that $\mathcal{A}_T\simeq [0,1]_{NM}$ or $\mathcal{A}_T\simeq [0,1]^\mathbb{Q}_{NM}$. From \Cref{theorem:nmstd,teo:nmtmp}, and with an argument similar to the one given in the proof of \Cref{teo:comp} we have that for each formula $\varphi$ and theory $\Gamma$:
\begin{equation*}
\Gamma\vdash_{NM}\varphi\quad\text{iff}\quad \Gamma\models_{T}\varphi.
\end{equation*}
\end{ex}
\section{Temporal semantics and rough sets}
In this section we discuss in more detail the connections between the temporal semantics previously described, and rough sets. As already pointed out in \Cref{sec:1}, from the mathematical point of view the algebras of the variety generated by the three elements Wajsberg-algebra $\mathbf{L}^w_3$ are term-equivalent with pre-rough algebras.
Moreover, in \cite{ban} a logical counterpart of pre-rough algebras, rough logic ($\mathcal{RL}$), has been introduced, by showing the following result:
\begin{theorem}[{\cite[Theorems 3.3, 3.5]{ban}}]
For every theory $\Gamma$ and formula $\varphi$ in the language of $\mathcal{RL}$ it holds that (here $\mathbf{R}$ represents the class of pre-rough algebras)
\begin{equation*}
\Gamma\vdash_{\mathcal{RL}}\varphi\qquad\text{iff}\qquad\Gamma\vdash_{\text{\L}_3}\varphi\qquad\text{iff}\qquad \Gamma\models_{\mathbf{L}^w_3}\varphi\qquad\text{iff}\qquad \Gamma\models_{\mathbf{R}}\varphi.
\end{equation*}
\end{theorem}
In particular, $\mathcal{RL}$ and \L$_3$ are defined over different sets of connectives: $\{\neg, \land, L\}$, and $\{\to,\neg\}$, respectively. However, in \cite{ban} it is shown that each set of connectives, as well as the associated logic, can be defined starting from the other, and two translations are described.

In the language of $\mathcal{RL}$, the necessity connective $L$ and its dual $M$ (defined as $M\varphi\df \neg L\neg\varphi$) are of particular interest. 
Indeed, they represents, intuitively, the lower and upper approximation of a set (described by a formula). According to the previously mentioned translations, in \L$_3$ the formula $L\varphi$ is defined as $\neg(\varphi\to\neg\varphi)$, whilst $M\varphi$ as $\neg L\neg\varphi$.

The semantics of these connectives, over $\mathbf{L}^w_3$, is depicted in \Cref{tab:2}.
\begin{table}[h]
\begin{center}
\begin{tabular}{|c|c|}
\hline
 $L_3$   &  \\
\hline
$0$  & $0$   \\
\hline
$\frac{1}{2}$  & $0$  \\
\hline
$1$  & $1$  \\
\hline
\end{tabular}
\hspace{2 cm}
\begin{tabular}{|c|c|}
\hline
 $M_3$   &  \\
\hline
$0$  & $0$   \\
\hline
$\frac{1}{2}$  & $1$  \\
\hline
$1$  & $1$  \\
\hline
\end{tabular}
\end{center}
\caption{Semantics of $L$ and $M$ over $\mathbf{L}^w_3$}
\label{tab:2}
\end{table}

As shown in the previous section, also NM is complete w.r.t. a semantics strictly connected with $\mathbf{L}^w_3$. The advantage of our temporal semantics is in its  expressive power, that allows to describe the behavior of formulas (in terms of their truth-values), over a temporal flow of time.

One can ask what we obtain if we define, over NM, two derived connectives $L$ and $M$ such that $L\varphi\df \neg(\varphi\to\neg\varphi)$ and $M\varphi\df\neg L\neg\varphi$. The answer is interesting, because a direct computation shows that, for every formula $\varphi$, $L\varphi$ is equivalent to $\varphi\&\varphi$, whilst $M\varphi$ is equivalent to $\varphi\veebar\varphi$. Note also that, due to \Cref{ex:l3}, we have that the semantics of $L$ and $M$ over the three element NM-chain coincide with the one of $\mathbf{L}^w_3$, and depicted in \Cref{tab:2}. In general, the semantics of $L$ and $M$ (we call $\mathcal{L}$ and $\mathcal{M}$ their algebraic counterparts), over an NM-chain, is the following:
\begin{align*}
\mathcal{L}x=&
\begin{cases}
0&\text{if }x\leq n(x)\\
x&\text{Otherwise.}
\end{cases}\\
\mathcal{M}x=&
\begin{cases}
1&\text{if }n(x)\leq x\\
x&\text{Otherwise.}
\end{cases}
\end{align*}
It can be easily checked from the results of \Cref{subsec:sem}. 

For more details about this topic, we refer to the forthcoming paper \cite{sh}. In this article, two operators defined like $L$ and $M$ are studied, for NM, as well as their logical and algebraic properties.

We now show the behavior of these two connectives over our temporal semantics. In particular, given a formula $\varphi$, and a temporal assignment $v$ (over some temporal flow) it holds that
\begin{align*}
v(L\varphi,t)=v(\neg(\varphi\to\neg\varphi),t)&=
\begin{cases}
\neg_3 (v(\varphi,t)\to_3 v(\neg\varphi,t))&\text{if }\neg_3 (v(\varphi,t)\to_3 v(\neg\varphi,t))=\\&\neg_3 (v(\varphi,t')\to_3 v(\neg\varphi,t')),\\
&\text{for every }t'\geq t\\
\frac{1}{2}&\text{otherwise.}
\end{cases}\\
v(M\varphi,t)=v(\neg\varphi\to\varphi,t)&=
\begin{cases}
v(\neg\varphi,t)\to_3 v(\varphi,t)&\text{if }v(\neg\varphi,t)\to_3 v(\varphi,t)=\\&v(\neg\varphi,t')\to_3 v(\varphi,t'),\\
&\text{for every }t'\geq t\\
\frac{1}{2}&\text{otherwise.}
\end{cases}
\end{align*}
Analogously to what happens for $\to$, also the temporal semantics of $L$ and $M$ is not truth-functional: this is not surprising, due to the way in which they are defined.
\section{Conclusions}
In this paper we have presented a three-valued temporal semantics for NM. It can be noted that, with some minor modifications, it can also be adapted for the logic NM$^-$, that is the logic corresponding to the variety of algebras generated by $[0,1]_\text{NM}\setminus \left\{\frac{1}{2}\right\}$ (\cite{gis,ct}).

Comparing the temporal semantics introduced in this paper with the one given for G\"{o}del logic in \cite{tmpg}, one of the main differences is that this last one is bivalent. From the technical point of view, it is possible to define a two-valued temporal semantics also for NM, but the problem is to find one that is not too much ``artificial'': future works could be addressed in this sense.

In this article we have also pointed out a first connection among NM and rough sets: however further directions of investigation are possible. For example, in \cite{bc} it is shown that NM is equivalent to constructive Nelson logic with strong negation (CLSN) plus prelinearity axiom \mbox{$(\varphi\to\psi)\vee(\psi\to\varphi)$}: since also pre-rough algebras are related with Nelson algebras (see for example \cite{pbook}), then future works could be addressed to analyze more in detail the connections among NM, pre-rough and Nelson algebras.
\section*{Acknowledgements}
The author would like to thanks to Professor Stefano Aguzzoli for the discussions and suggestions about the topic of this paper, that have furnished the original idea of this temporal semantics, and have contributed in a significant way to the development of this article.

\noindent Also, the author would like to thank to the reviewers, that with their suggestions have contributed to improve the readability and the content of the paper. 
\bibliographystyle{amsalpha}
\bibliography{temporalNM}
\end{document}